\documentclass[11pt]{amsart}

\usepackage[utf8]{inputenc}
\usepackage[english]{babel}
\usepackage[T1]{fontenc}

\title[Generic level sets in MCF]{Generic level sets in mean curvature flow are BV solutions}

\thanks{This preprint has not undergone peer review. There is a published version of this article in the Journal of Geometric Analysis, which is available online at \url{https://doi.org/10.1007/s12220-024-01819-y}}

\author{Anton Ullrich}
\thanks{A.U.\ (corresponding author) Max Planck Institute for Mathematics in the Sciences, Inselstraße 22, 
04103 Leipzig,
Germany. \nolinkurl{anton.ullrich@mis.mpg.de}}

\author{Tim Laux}
\thanks{T.L.\ Hausdorff Center for Mathematics and Institute for Applied Mathematics, University of Bonn, Villa Maria, Endenicher Allee 62, 53115 Bonn, Germany. \nolinkurl{tim.laux@hcm.uni-bonn.de}}

\date{}

\usepackage{amsthm}
\usepackage{stmaryrd}
\usepackage{enumitem}

\usepackage{algorithm2e}
\SetKwComment{Comment}{/* }{ */}
\RestyleAlgo{ruled}
\usepackage{listings}
\usepackage{verbatim}
\usepackage{stmaryrd}
\usepackage{dsfont}
\usepackage{tikz}
\usetikzlibrary{calc, quotes, angles}
\usepackage{float}
\usepackage{amsmath}
\usepackage{amssymb}
\usepackage{amsthm}
\usepackage{mathrsfs}

\usepackage{mathtools}
\usepackage{esint}
\usetikzlibrary{decorations.pathreplacing}

\PassOptionsToPackage{dvipsnames}{xcolor}
\RequirePackage{xcolor}
\definecolor{webgreen}{rgb}{0,.5,0}
\definecolor{webbrown}{rgb}{.6,0,0}
\definecolor{RoyalBlue}{cmyk}{1, 0.50, 0, 0}

\usepackage[colorlinks,
pdfpagelabels,
pdfstartview = FitH,
bookmarksopen = true,
bookmarksnumbered = true,
linkcolor = RoyalBlue,
plainpages = false,
urlcolor=webbrown,
hypertexnames = false,
citecolor = webgreen] {hyperref}

\newcommand*{\dd}{\mathop{}\!\mathrm{d}}

\DeclareMathOperator*{\esssup}{ess\,sup}

\theoremstyle{plain}
\newtheorem{defi}{Definition}[section]

\newtheorem*{thm*}{Theorem}
\newtheorem*{ex*}{Example}
\newtheorem*{defi*}{Definition}
\newtheorem*{remark*}{Remark}
\newtheorem*{prop*}{Proposition}
\newtheorem{thm}[defi]{Theorem}

\newtheorem{lemma}[defi]{Lemma}
\newtheorem{remark}{Remark}
\newtheorem*{lemma*}{Lemma}

\usepackage{pgfplots}
\pgfplotsset{compat=newest}
\usetikzlibrary{intersections, pgfplots.fillbetween}
\pgfdeclarelayer{pre main}
\allowdisplaybreaks

\begin{document}

\begin{abstract}
    We show that a generic level set of the viscosity solution to mean curvature flow is a distributional solution in the framework of sets of finite perimeter by Luckhaus and Sturzenhecker, which in addition saturates the optimal energy dissipation rate.
    This extends the fundamental work of Evans and Spruck (J.\ Geom.\ Anal.\ 1995), which draws a similar connection between the viscosity solution and Brakke flows.
\end{abstract}


\maketitle

\section{Introduction}

The mean curvature flow is a parabolic geometric evolution equation with numerous applications in the sciences and engineering.
It is the most basic model describing slow relaxation driven by surface tension and can be viewed as the diffusion equation of surfaces relating the normal velocity $V$ and the mean curvature $H$ of an evolving surface via the equation $V=-H$.

\medskip

Since the mean curvature flow develops singularities in finite time, it is natural to consider weak solutions which are able to describe the evolution through singular events such as pinch-off. 
However, after such singularities, the flow may become non-unique~\cite{AngenentChoppIlmanen} and it is in general difficult to compare different weak solutions.
The aim of this paper is to better understand weak solutions and in particular draw a new connection between different weak solution concepts.
We will show that generically, they are consistent with each other.

\medskip

More precisely, we consider the viscosity solution introduced independently by Evans and Spruck~\cite{ESI}, and Chen, Giga and Goto~\cite{ChenGigaGoto}.
This concept is based on the level set formulation of Osher and Sethian~\cite{OsherSethian} in which all level sets of a function are evolved simultaneously (Figure~\ref{Fig:LevelSetFct}). 
We show that almost every level set of the viscosity solution is an evolving set of finite perimeter which moves according to the distributional formulation of mean curvature flow due to Luckhaus and Sturzenhecker~\cite{LS}. 
In addition, we show that the evolving sets of finite perimeter satisfy an optimal energy dissipation relation, which is crucial for the weak-strong uniqueness of the solution concept, cf.~\cite{FHLS}. 
Additionally, together with the distributional identity between normal velocity and mean curvature operator, the solution can be viewed as curve of maximal slope in the spirit of De~Giorgi, cf.~\cite{DeGiorgiMarinoTosques,AmbrosioGigliSavareBook}. 
More precisely, our solution also recovers the De~Giorgi formulation in the setting of BV solutions to mean curvature flow which was introduced by Otto and one of the authors~\cite{LauxOttoThreshMCFDeGiorgiMM} to describe limits of the thresholding scheme. 

\medskip

It is a by now classical result due to Evans and Spruck~\cite{ESI} that the viscosity solution is consistent with the classical solution in the sense that, when starting from the same initial data, the two solutions coincide on their common time interval of existence. 
More recently, such a consistency --- or weak-strong uniqueness --- has also been shown between BV solutions and the classical solution to mean curvature flow in~\cite{FHLS}.
Our result extends those statements by stating that also past singularities, when the classical solution ceases to exist, these two weak solution concepts are consistent with each other.  
More precisely, we prove the following statement.
\begin{thm}
Let $u\in C(\mathbb{R}^d\times [0,T))$ be a viscosity solution to mean curvature flow in the sense of Definition~\ref{Def:ViscositySolution} with well-prepared initial data according to Definition~\ref{def:wellprepared}. 
Then almost every level set of $u$ is a BV solution to mean curvature flow in the sense of Definition~\ref{Def:BVsol}.
\label{Thm:MainThm}
\end{thm}

The idea of our proof is rather simple. 
We use the vanishing-viscosity approximation by Evans and Spruck~\cite{ESI}. 
The strengthened convergence necessary to carry out our argument comes from the surprising estimate from~\cite{ESIV} in the spirit of
\[
    \sup_{t>0} \int\limits_{\mathbb{R}^d} |H| \dd x <\infty.
\]
This makes possible a compensated compactness argument to show that the energies, which are modeled after
$\int_{\mathbb{R}^d} |\nabla u| \dd x$,
converge in the vanishing-viscosity limit.
By lower semi-continuity and the coarea-formula, this implies that almost every level set converges strictly, i.e., the perimeter does not drop down in the limit. 
Such an energy convergence is an essential tool for convergence proofs for approximations of mean curvature flow, see~\cite{LS, LauxOttoThreshMPMCF,LauxSimon}.
Therefore, it is a natural question whether one can verify the BV formulation also for the viscosity solution of the level set flow.
We do this by first deriving a distributional formulation of level set mean curvature flow that is satisfied by the viscosity solution. 
Then we disintegrate the level set parameter in this formulation. On a technical level, this is done by using the coarea-formula for surface-type integrals and the layer-cake formula for volume-type integrals.

\section{Definitions of the distributional and generic formulation}
\label{DistForm}

In this section, we recall the definition of the level set solution to mean curvature flow and how to generalize this concept to viscosity solutions. Afterwards, the notion of distributional solutions is stated and explained.

\begin{figure}[!ht]
    \centering
    \includegraphics[scale=0.3]{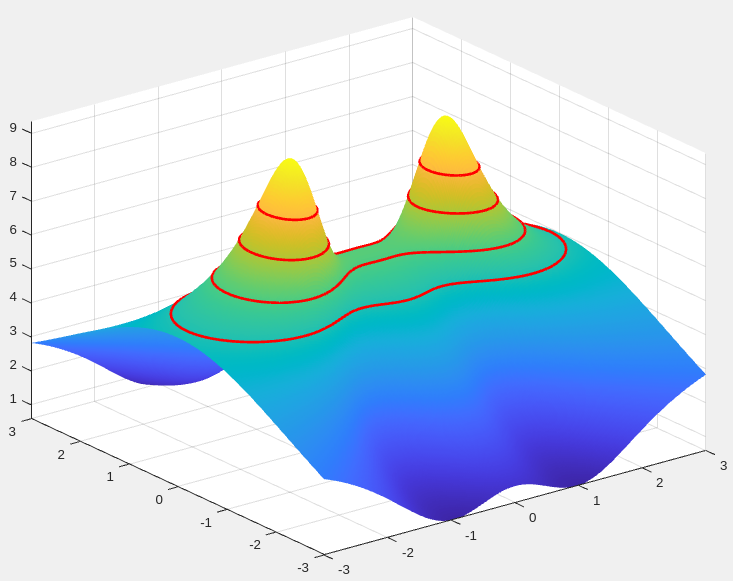}
    \caption{Example of a time-instance of a level set function. Level sets for four different values are marked in red.}
    \label{Fig:LevelSetFct}
\end{figure}

To give an intuition of the level set equation of mean curvature flow, let $u\colon\mathbb{R}^d\times [0,T)\to \mathbb{R}$ be a smooth function such that each level set evolves by mean curvature flow. 
We aim to find a partial differential equation for $u$. 
To this end, let $x(t)\in\mathbb{R}^d$, $t\in (t_1,t_2)\subset(0,T)$, be a trajectory following one level set, i.e, $u(x(t),t) = \textup{const.}$ and suppose $\nabla u(x(t),t) \neq 0$ for all $t\in(t_1,t_2)$. 
The outer unit normal vector to the super level sets is given by $\nu(x,t)=-\frac{\nabla u(x,t)}{|\nabla u(x,t)|}$ and since the mean curvature is the divergence of the outer unit normal, $H=\nabla\cdot \nu$, the mean curvature flow equation can be expressed by $\dot x\cdot \nu=-\nabla\cdot \nu$. Thus
\begin{align*}
0=\frac{\dd}{\dd t}u(x(t),t)
=\partial_t u+\dot x\cdot\nabla u
=\partial_t u+|\nabla u|\nabla\cdot\nu
\end{align*}
and hence, we obtain the level set equation for $u$:
$$\partial_t u=-|\nabla u|\nabla\cdot \nu=\Delta u-\frac{\nabla u}{|\nabla u|}\cdot \nabla^2 u\frac{\nabla u}{|\nabla u|}.$$

This leads to the following definition which is equivalent to the fact that the level sets of $u\in C^2$ evolve by mean curvature flow in regions where the gradient does not vanish, cf.~\cite{ESI}.

\begin{defi}[level set equation] A function $u\colon\mathbb{R}^d\times (0,T)\to\mathbb{R}$ with $ u\in C^1_t\cap C^2_x$ is called a solution to the \emph{level set equation} of mean curvature flow if it solves the initial value problem
\begin{align}
    \begin{cases}\partial_t u=\Delta u-\frac{\nabla u}{|\nabla u|}\cdot \nabla^2 u\frac{\nabla u}{|\nabla u|}&\text{in }\mathbb{R}^d\times (0,T), \{\nabla u\neq 0\},\\\partial_t u=0&\text{in }\mathbb{R}^d\times (0,T), \{\nabla u=0\},\\u=g &\text{on }\mathbb{R}^d\times\{t=0\}.\end{cases}
    \label{Eq:lvleq}
\end{align}
\end{defi}

This solution concept can be extended to the non-smooth case even after the onset of singularities using viscosity solutions, cf.~\cite{ESI, ChenGigaGoto}. These use the comparison principle of the mean curvature flow to control the evolution of the level sets. Here, one considers continuous level set functions. Intuitively, the idea is to take smooth functions locally approximating the level set function in an ordered way and impose the respective inequality on these smooth functions. This solution concept is consistent with the level set solutions for smooth functions $u$. One way to prove this is by the characterization using sub- and superjets, cf.~\cite[Section 2.3]{ESI}.

\begin{defi}[Viscosity solution] \label{Def:ViscositySolution}
A continuous function $u\in C(\mathbb{R}^d\times[0,T))$ is called a \emph{viscosity super-solution} to \eqref{Eq:lvleq} if, for any $\varphi\in C^\infty(\mathbb{R}^d\times(0,T))$ and $(x_0,t_0)\in \mathbb{R}^d\times(0,T)$ such that $u-\varphi$ has a local minimum at $(x_0,t_0)$, the following inequality holds,
\begin{enumerate}[label=\roman*)]
    \item for non-critical points $\nabla \varphi(x_0,t_0)\neq 0:$
    $$\partial_t \varphi\geq \Delta \varphi-\frac{\nabla \varphi}{|\nabla \varphi|}\cdot \nabla^2\varphi\frac{\nabla \varphi}{|\nabla \varphi|}\quad\text{at }(x_0,t_0)$$
    \item and for critical points $\nabla \varphi(x_0,t_0)=0$ there exists $\xi\in\mathbb{R}^d, |\xi|\leq 1:$
    $$\partial_t \varphi\geq \Delta \varphi-\xi\cdot \nabla^2\varphi \xi\quad\text{at }(x_0,t_0).$$
\end{enumerate}
We say that $u$ is a viscosity sub-solution if $-u$ is a viscosity super-solution. Finally, $u$ is called a viscosity solution if it is a viscosity sub- as well as a super-solution.
\end{defi}

Viscosity solutions are unique, i.e., it can be shown that there exists at most one viscosity solution to mean curvature flow, cf.~\cite[Theorem 3.2]{ESI}. However, the level sets may develop an interior, a phenomenon called fattening. This is analyzed, for example, in~\cite{angenent2002fattening}.

\medskip

Another way to extend the solution concept of mean curvature flow are so-called BV solutions due to Luckhaus and Sturzenhecker~\cite{LS}.
This definition uses distributional characterizations of the mean curvature and the normal velocity. 
Additionally, we impose a sharp energy dissipation which extends the definition of Luckhaus and Sturzenhecker.

In the following the Gau{\ss}-Green measure $\mu_\Omega$ denotes the (negative) spatial distributional derivative of the characteristic function $\mu_\Omega(t)=-\nabla \chi_{\Omega(t)}$ for every time $t\in (0,T)$ and $|\mu_\Omega(t)|$ denotes its total variation. Moreover, the perimeter of a set $\Omega(t)$ is defined as
$$P(\Omega(t)):=|\mu_\Omega(t)|(\mathbb{R}^d)=\sup\left\{\int\limits_{\Omega(t)}\nabla\cdot \xi\dd x\colon \xi\in C_c^1(\mathbb{R}^d;\mathbb{R}^d), |\xi|\leq 1\right\}.$$
Note that by De~Giorgi's structure theorem (cf.\ Theorem~\ref{DeGiorgiStructure}), integration against the Gau{\ss}-Green measure precisely corresponds to integration over the reduced boundary $\partial^* \Omega(t)$ with respect to the Hausdorff measure $\mathcal{H}^{d-1}$.

\begin{defi}[BV solution]
\label{Def:BVsol}
A family of finite perimeter sets $(\Omega(t))_{t\in [0, T)}$ s.t.\ $\bigcup\limits_{t\in [0,T)}\Omega(t)$ is measurable in space and time is called \emph{distributional solution to mean curvature flow} if the sets admit a uniform perimeter bound
\begin{align}
    \esssup\limits_{t\in [0,T)}P(\Omega(t))&<\infty\label{Eq:LSP}
\end{align}
and there is a $\dd |\mu_\Omega(t)|\dd t$-measurable $L^2$-bounded function $V\colon\mathbb{R}^d\times (0,T)\to \mathbb{R}$
\begin{align}
    \int\limits_{\mathbb{R}^d\times (0,T)}V^2\dd |\mu_\Omega(t)|\dd t&<\infty\label{Eq:LSV}
\end{align}
such that the following holds:
\begin{enumerate}[label=\roman*)]
    \item The function $V(\cdot, t)$ is the normal velocity of $\Omega(t)$ in the sense that for each test function $\zeta\in C_c^1(\mathbb{R}^d\times [0,T))$ it satisfies
    \begin{align}
        \int\limits_0^T\int\limits_{\Omega(t)}\partial_t\zeta\dd x\dd t&=-\int\limits_{\mathbb{R}^d\times(0,T)}\zeta V\dd |\mu_\Omega(t)|\dd t-\int\limits_{\Omega(0)}\zeta(x, 0)\dd x\label{Eq:LSdistV}.
    \end{align}
    \item Furthermore, for any test vector field $\xi\in C_c^1(\mathbb{R}^d\times(0,T);\mathbb{R}^d)$ it holds
    \begin{align}
        \int\limits_{\mathbb{R}^d\times(0,T)} (\nabla\cdot \xi-\nu\cdot \nabla \xi\nu) \dd |\mu_\Omega(t)|\dd t&=-\int\limits_{\mathbb{R}^d\times(0,T)} V\xi\cdot \nu\dd |\mu_\Omega(t)|\dd t.\label{LSdistMC}
    \end{align}
    \item For almost any time $T'\in (0, T)$ the sharp energy-dissipation holds
    \begin{align}
        P(\Omega(T'))+\int\limits_{\mathbb{R}^d\times (0, T')}V^2\dd |\mu_\Omega(t)|\dd t\leq P(\Omega(0)).\label{def:BVEnergyDissipation}
    \end{align}
\end{enumerate}
\end{defi}

Equation~\eqref{Eq:LSdistV} characterizes the normal velocity for smooth sets $\Omega(t)$, which is a direct consequence of the computation of $\partial_t\int_{\Omega(t)}\zeta\dd x$.
The operator on the left-hand side of Equation~\eqref{LSdistMC} encodes the mean curvature via an integration by parts on the surface, which can be easily seen for smooth sets.

This solution concept represents the mean curvature flow since $V$ is the distributional normal velocity and simultaneously the negative distributional mean curvature.

In addition to the distributional formulation by Luckhaus and Sturzenhecker, we also have a sharp energy-dissipation relation given by Inequality~\eqref{def:BVEnergyDissipation}.

\section{Recap of results by Evans and Spruck}

In this section, we recall the important results from \cite{ESI}, \cite{ESIII} and \cite{ESIV} that are needed to prove Theorem~\ref{Thm:MainThm}. We will consider the vanishing viscosity approximation to our level set formulation
\begin{align}
    \partial_t u_\varepsilon&=\Delta u_\varepsilon-\frac{\nabla u_\varepsilon}{\sqrt{|\nabla u_\varepsilon|^2+\varepsilon^2}}\cdot \nabla^2 u_\varepsilon\frac{\nabla u_\varepsilon}{\sqrt{|\nabla u_\varepsilon|^2+\varepsilon^2}}.\label{epsAppr}
\end{align}
This approximation is non-degenerate and has sufficient convergence to the level set solution as we will see in Theorem~\ref{conv}. In addition, it has a beautiful geometric interpretation as the mean curvature flow of graphs in a higher dimensional space, cf.~\cite{ESI}.

Throughout, we will assume that the initial datum for the  viscous approximation~\eqref{epsAppr} is well-prepared in the following sense.
\begin{defi}\label{def:wellprepared}
We call an initial level set data $g$ \emph{well-prepared} if $g\in C^3(\mathbb{R}^d)$, $g$ is constant outside a ball of radius $R>0$, and $g$ satisfies
\begin{align}
    \sup_{0<\varepsilon\leq 1} \int\limits_{\mathbb{R}^d} \Big|\nabla \cdot \Big(\frac{\nabla g}{\sqrt{|\nabla g|^2 + \varepsilon^2}}\Big) \Big| \dd x <\infty.
\end{align}
\end{defi}

\begin{remark}
We will call this integrand approximate mean curvature of $g$ and say that $g$ possesses an integrable approximate mean curvature with a uniform bound in $\varepsilon$ if the above bound holds.

Following~\cite[Lemma 2.1]{ESIV} one can construct well-prepared initial data $g$ for any given compact smooth $0$-level set.
\end{remark}

The viscous approximation \eqref{epsAppr} has better properties in the sense that it yields smooth solutions, cf.~\cite[Chapter 2]{ESIII}. For these, the following bounds on the solution and its derivatives hold.

\begin{lemma}[{\cite[Theorem 4.1]{ESI}}]
For well-prepared initial data, there exists a unique smooth solution $u_\varepsilon$ to \eqref{epsAppr}. Additionally, $|u_\varepsilon|,|\nabla u_\varepsilon|,|\partial_t u_\varepsilon|$ are uniformly bounded in $\varepsilon, x$ and $t$. \label{epsSol}
\end{lemma}

The central quantity of our argument will be the generalized normal velocity
\begin{equation}\label{eq:defV}
V:=\frac{\partial_t u}{|\nabla u|}\chi_{\{\nabla u\neq 0\}}.
\end{equation}
We will also need the generalized outer unit normal vector
$$\nu:=\begin{cases}-\frac{\nabla u}{|\nabla u|} & \text{if }\nabla u\neq 0,\\e_1 &\text{else.}\end{cases}$$
To prove the convergences of the vanishing viscosity approximation, the most crucial step is to utilize the following unexpected estimate due to Evans and Spruck~\cite{ESIV}.
\begin{thm}[Theorem 2.2 of \cite{ESIV}]
We have the uniform $L^\infty_tL^1_x$-bound
\begin{align}
    \sup\limits_{0\leq t\leq T}\sup\limits_{0<\varepsilon\leq 1}\int\limits_{\mathbb{R}^d}|H_\varepsilon (x,t)|\dd x<\infty,\label{strangeEquation}
\end{align}
where $H_\varepsilon(x,t):=\nabla\cdot\nu_\varepsilon$ is called the approximate mean curvature for the level sets of $u^\varepsilon$ and $\nu_\varepsilon=-\frac{\nabla u_\varepsilon}{\sqrt{|\nabla u_\varepsilon|^2+\varepsilon^2}}$ the approximate outer normal vector.
\label{strangeEst}
\end{thm}
This statement can be found in \cite[Theorem 2.2]{ESIV} and follows from the fact that the quantity $t\mapsto \int_{\mathbb{R}^d} |H_\varepsilon(x,t)| \dd x$ is non-increasing.
This uniform $L^\infty_tL^1_x$-bound on the approximate mean curvature has strong consequences as can be seen in the following theorem, cf.~\cite[Chapter 3]{ESIV}. For a subsequence, the approximate mean curvature $H_\varepsilon$ converges in the weak-$*$ topology of measures. Together with the relation to the approximate normal this yields the additional convergences in the following.

One can show that $u$ is bounded and Lipschitz in space and time, cf.~\cite{ESIII}. Thus, by Rademacher's theorem the derivatives are defined almost everywhere.

\begin{thm}[Convergences, {cf.~\cite[Section~2]{ESIII}, \cite[Section 3]{ESIV}}]
\label{conv}
For the vanishing viscosity solution $u_\varepsilon$, we have the following convergences:
\begin{align*}
    u_\varepsilon&\to u \text{ locally uniformly in }\mathbb{R}^d\times [0,T),\\
    \nabla u_\varepsilon&\overset{*}{\rightharpoonup} \nabla u\text{ in }L^\infty(\mathbb{R}^d\times [0,T)),\\
    \partial_t u_\varepsilon&\overset{*}{\rightharpoonup} \partial_t u\text{ in }L^\infty(\mathbb{R}^d\times [0,T)).
\end{align*}
Additionally, the following convergences also hold
\begin{align*}
    |\nabla u_\varepsilon|&\overset{*}{\rightharpoonup} |\nabla u|\text{ in }L^\infty(\mathbb{R}^d),\\
    \nu_\varepsilon&\overset{*}{\rightharpoonup} \nu\text{ in }L^\infty(\{\nabla u\neq 0\}),\\
    \nu_\varepsilon&\to\nu\text{ in }L^2(B_R\cap \{\nabla u\neq 0\})\forall R>0.
\end{align*}
\end{thm}

The first assertions follow from Arzel\`a-Ascoli and Banach-Alaoglu under a subsequence when using Lemma~\ref{epsSol}.

The key argument for the rest of the theorem is that once one writes
\begin{align*}
    |\nabla u_\varepsilon| \approx - \nu_\varepsilon \cdot\nabla u_\varepsilon,
\end{align*}
one can pass to the limit in the product on the right-hand side via a compensated compactness argument using Theorem~\ref{strangeEst} since $u_\varepsilon \to u$ uniformly and $\nabla \cdot \nu_\varepsilon = H_\varepsilon$ is uniformly bounded in $\varepsilon$ as Radon measures.
Another useful lemma is the following, which tells us that certain integral quantities are bounded. It will be important in taking the limit and proving estimates for $u$ given the corresponding estimates for $u_\varepsilon$.

This is Theorem 4.1 in \cite{ESIV} and follows from the bounds of $u_\varepsilon$ and its derivatives (Lemma~\ref{epsSol}).
\begin{lemma}[{\cite[Theorem 4.1]{ESIV}}]
Let $K\subset \mathbb{R}^d\times(0,T)$ be a compact set. Then the following bound on the approximate mean curvature holds
$$\sup\limits_{0<\varepsilon\leq 1}\int\limits_0^T\int\limits_{\mathbb{R}^d}\chi_{K}|H_\varepsilon|^2\sqrt{|\nabla u_\varepsilon|^2+\varepsilon^2}\dd x\dd t<\infty.$$
\label{HsqrtEst}
\end{lemma}
This estimate is more intuitive than \eqref{strangeEquation} as for smooth manifolds evolving by mean curvature flow $\int |H|^2$ is the first variation of the perimeter functional and the factor $|\nabla u|$ is the coarea-factor.

Additionally, we will use a relabeling property which is stated in the following lemma and originates from~\cite[Proof of Theorem 6.3]{ESIV}.

\begin{lemma}[Relabeling property, cf.~{\cite[Proof of Theorem 6.3]{ESIV}}]
    Let $g$ be well-prepared initial data, cf.\ Definition~\ref{def:wellprepared} and $u$ the unique viscosity solution to~\eqref{Eq:lvleq} as well as $\Phi\in C^\infty(\mathbb{R})$ be a smooth test function with positive derivative $\Phi'>0$. Then $\Phi\circ u$ is the unique viscosity solution to the initial data $\Phi\circ g$. Moreover, the initial data $\Phi\circ g$ is well-prepared.
    \label{Relabeling}
\end{lemma}

\section{Generic level sets are BV solutions}

Next, we will begin with the proof of Theorem~\ref{Thm:MainThm}, i.e., that almost every level set of the viscosity solution is a distributional solution in the sense of BV solutions.

The following first theorem can be understood as an integrated version of the distributional equations for normal velocity and mean curvature. 
The main novelty is Equation~\eqref{1distV}, which encodes the normal velocity while the second part (Equation~\eqref{1distMC}) can be found as Theorem 5.1 in \cite{ESIV} where the name $H$ instead of $V$ is used.
Having proved these identities for the level set function, one can then proceed by disintegrating the level set parameter to get the corresponding identities for almost every level set. 
This is done below in Theorem~\ref{lvlVariants} using the layercake-formula and the coarea-formula (Lemma~\ref{coareaformula}).

\begin{thm}
\label{distViscosity}
Let $u\in C(\mathbb{R}^d\times[0,T))$ be the unique viscosity solution with some well-prepared initial data $g$.
Then there exists a $|\nabla u|\dd x \dd t$-measurable function $V\colon \mathbb{R}^d\times(0,T)\to\mathbb{R}$ such that
\begin{enumerate}[label=\roman*)]
    \item For any test function $\zeta\in C_c^1(\mathbb{R}^d\times [0,T))$ it holds
    \begin{align}
        \int\limits_0^T\int\limits_{\mathbb{R}^d} \partial_t\zeta u\dd x\dd t
        &=-\int\limits_0^T\int\limits_{\mathbb{R}^d} \zeta V|\nabla u|\dd x\dd t - \int\limits_{\mathbb{R}^d} g\,\zeta(\cdot,0)\dd x.\label{1distV}
    \end{align}
    \item For any test vector field $\xi\in C_c^1(\mathbb{R}^d\times (0,T);\mathbb{R}^d)$ we have
    \begin{align}
       \int\limits_0^T\int\limits_{\mathbb{R}^d}\left(\nabla \cdot \xi-\nu\cdot \nabla\xi\nu\right)|\nabla u|\dd x\dd t&=-\int\limits_0^T\int\limits_{\mathbb{R}^d} \xi\cdot\nu V|\nabla u|\dd x\dd t.\label{1distMC}
    \end{align}
\end{enumerate}
\end{thm}

\begin{proof}
The idea is to prove an approximate version for $u_\varepsilon$ using integration by parts and the definition of the vanishing viscosity equation. Afterwards, one uses Lemma~\ref{HsqrtEst} and the convergences in Theorem~\ref{conv} to get the equation for $u$. Here, the convergence up to the set $\{\nabla u=0\}$ suffices since their contribution will vanish.

\emph{Argument for~\eqref{1distV}.} 
Fix a test function $\zeta\in C_c^1(\mathbb{R}^d\times [0,T))$. 
We first claim that $\partial_t u = V |\nabla u|$ almost everywhere on $\mathbb{R}^d\times(0,T)$. 
Then, we may conclude via
\begin{align*}
    \int\limits_0^T\int\limits_{\mathbb{R}^d}\partial_t\zeta u\dd x\dd t&=-\int\limits_0^T \int\limits_{\mathbb{R}^d}\zeta \partial_t u\dd x\dd t-\int\limits_{\mathbb{R}^d}g\zeta(\cdot,0)\dd x\\
    &=-\int\limits_0^T \int\limits_{\mathbb{R}^d}\zeta V|\nabla u|\dd x\dd t-\int\limits_{\mathbb{R}^d}g\,\zeta(\cdot,0)\dd x,
\end{align*}
which is precisely~\eqref{1distV}.

Now, we argue that indeed $\partial_t u = V |\nabla u|$ almost everywhere. 
First, observe that by definition of the normal velocity $V$, cf.~\eqref{eq:defV}, the identity holds on the set $\{\nabla u \neq 0\}$.
Thus, we only need to argue that $\partial_tu=0$ almost everywhere on $\{\nabla u=0\}$. 
This in turn follows from the next computation, cf.~\cite[Lemma 4.2]{ESIV}. 
Note that we can express the evolution equation, Equation~\eqref{epsAppr}, as $\partial_tu_\varepsilon=-H_\varepsilon\sqrt{|\nabla u_\varepsilon|^2+\varepsilon^2}$. Then, for every compact subset $A$ of the critical set $\{\nabla u=0\}$:
\begin{align*}
    \left|\int\limits_A \partial_t u\dd x\dd t\right|&\leq \lim\limits_{\varepsilon\searrow 0}\int\limits_A |\partial_tu_\varepsilon|\dd x\dd t\\
    &=\lim\limits_{\varepsilon\searrow 0}\int\limits_A |H_\varepsilon|\sqrt{|\nabla u_\varepsilon|^2+\varepsilon^2}\dd x\dd t\\
    &\leq \liminf\limits_{\varepsilon\searrow 0}\left(\int\limits_A H_\varepsilon^2\sqrt{|\nabla u_\varepsilon|^2+\varepsilon^2}\dd x\dd t\right)^\frac{1}{2}\\
    &\quad\quad\quad\quad\quad  \times\left(\int\limits_A \sqrt{|\nabla u_\varepsilon|^2+\varepsilon^2}\dd x\dd t\right)^\frac{1}{2}\\
    &\leq C\left(\int\limits_A|\nabla u|\dd x\dd t\right)^\frac{1}{2}=0,
\end{align*}
where in the last step we have used Lemma~\ref{HsqrtEst} to bound the first term and Theorem~\ref{conv} to pass to the limit in the second term. 
Now, by inner regularity this proves $\partial_t u=0$ almost everywhere on $\{\nabla u=0\}$.

\emph{Argument for~\eqref{1distMC}.} As stated above, the following argument can be found in \cite[Theorem 5.1]{ESIV} but for the convenience of the reader, we highlight the main idea. 
Testing Equation~\eqref{epsAppr} with $\xi \cdot \nu_\varepsilon$ for a given test vector field $\xi\in C_c^1(\mathbb{R}^d\times(0,T); \mathbb{R}^d)$  and integrating by parts yields the following $\varepsilon$-variant of~\eqref{1distMC}:
\begin{align*}
   \int\limits_0^T\int\limits_{\mathbb{R}^d} \left(\nabla\cdot\xi-\nu_\varepsilon\cdot \nabla\xi \nu_\varepsilon\right)\sqrt{|\nabla u_\varepsilon|^2+\varepsilon^2}\dd x\dd t= -\int\limits_0^T\int\limits_{\mathbb{R}^d}&\xi\cdot \nu_\varepsilon \partial_t u_\varepsilon \dd x\dd t.
\end{align*}
For a detailed derivation, cf.\ Section~\ref{homogNeumann}.
Now, Theorem~\ref{conv} and Lemma~\ref{HsqrtEst} allow us to pass to the limit $\varepsilon\searrow0$ on both sides of this identity.
\end{proof}

We additionally need the optimal energy dissipation. This will be a consequence of the following theorem for the level set function $u$, which follows from \cite[Theorem 5.2]{ESIV} where again $V$ is called $H$.

\begin{thm}
\label{DeGiorgiViscosity}
For the viscosity solution $u$ and any two time instances $0\leq t_1<t_2\leq T$, we have a version of the optimal energy dissipation:
\begin{align}
    \left.\int\limits_{\mathbb{R}^d}|\nabla u|\dd x\right|_{t_2}+\int\limits_{t_1}^{t_2}\int\limits_{\mathbb{R}^d}V^2|\nabla u|\dd x\dd t\leq \left.\int\limits_{\mathbb{R}^d}|\nabla u|\dd x\right|_{t_1}.\label{phiDeriv}
\end{align}
\end{thm}

From the previous two theorems, one can extract the level set case for almost every level set by use of the coarea-formula (Lemma~\ref{coareaformula}) and the layercake-formula. To this end, it is necessary to define a notion for the level sets and super level sets.

\begin{defi}[$\Omega_s, \Sigma_s$] For a viscosity solution $u$ to \eqref{Eq:lvleq} and any $s\in \mathbb{R}$ and $t\in[0,\infty)$, one defines the super level sets $\Omega_s(t):=\{x:u(x,t)>s\}$ and the level sets $\Sigma_s(t):=\{x:u(x,t)=s\}$.
\label{OmegaSigma}
\end{defi}

The following theorem states the disintegrated version of Theorem~\ref{distViscosity}.

\begin{thm}
Let $u$ be a viscosity solution to the level set equation~\eqref{Eq:lvleq} and let $\Omega_s, \Sigma_s$ be given as in Definition~\ref{OmegaSigma}. Then, for almost every value $s\in \mathbb{R}$, we have
\begin{enumerate}[label=\roman*)]
    \item  The normal velocity of $\Sigma_s$ is given by $V$ in the sense that for all test functions $\zeta\in C_c^1(\mathbb{R}^d\times (0,T))$:
    \begin{align}
        \int\limits_0^T \int\limits_{\Omega_s(t)}\partial_t\zeta\dd x\dd t&=-\int\limits_0^T\int\limits_{\Sigma_s(t)}\zeta V\dd\mathcal{H}^{d-1}(x)\dd t-\int_{\Omega_s(0)} \zeta(\cdot,0)\dd x,\label{lvlV}
    \end{align}
    \item The equation $V=-H$ is satisfied for $\Sigma_s$ in the sense that for all test vector fields $\xi\in C_c^1(\mathbb{R}^d\times (0,T);\mathbb{R}^d)$:
    \begin{align}
        \int\limits_0^T\int\limits_{\Sigma_s(t)}(\nabla\cdot\xi-\nu\cdot\nabla\xi\nu)\dd\mathcal{H}^{d-1}(x)\dd t&=-\int\limits_0^T\int\limits_{\Sigma_s(t)}\xi\cdot \nu V\dd\mathcal{H}^{d-1}(x)\dd t.\label{lvlMC}
    \end{align}
\end{enumerate}
\label{lvlVariants}
\end{thm}

Equation~\eqref{lvlMC} is similar to the one found in \cite{ESIV} now on a level set basis while \eqref{lvlV} is new.

\begin{proof}
The idea is to use Theorem~\ref{distViscosity} combined with the relabeling property. By the coarea-formula and layercake-formula, one can split the integrals into the (super-)level sets and then use the relabeling to separate these and get the equality for almost every level set separately.

\emph{Argument for~\eqref{lvlV}}: Fix $\Phi\in C^\infty(\mathbb{R})$ with $\Phi'>0$ and observe that $\Phi\circ u$ is the unique viscosity solution to the well-prepared initial data $\Phi\circ g$ by Lemma~\ref{Relabeling}. 
Now, we can use Equation~\eqref{1distV} with $V_\Phi=\frac{\partial_t u}{|\nabla u|}=V$ in $\{|\nabla u|\neq 0\}$. 
By the coarea-formula, the right-hand side is equal to 
\begin{align*}
-\int\limits_\mathbb{R}\Phi'(s)\int\limits_0^T\int\limits_{\Sigma_s(t)} \zeta V\dd\mathcal{H}^{d-1}(x)\dd t\dd s- \int\limits_{\mathbb{R}^d} g\zeta(\cdot,0)\dd x.
\end{align*}
For the left-hand side and the initial data term, one can use the layercake-formula with a transformation, where $K\in\mathbb{R}, K<u_{\text{min}}$ and $u_{\text{min}}\coloneqq \inf\limits_{(x,t)}u(x,t)$ is the infimum of $u$ in space and time (bounded by~Lemma~\ref{epsSol}):
\begin{align*}
    \Phi\circ u
    =\int\limits_{\Phi(K)}^{\Phi\circ u}1\dd \tau+\Phi(K) =\int\limits_{K}^{\infty}\chi_{\{u>s\}}\Phi'(s)\dd s+\Phi(K).
\end{align*}
This yields for the left-hand side of \eqref{1distV}
\begin{align*}
    \int\limits_0^T\int\limits_{\mathbb{R}^d} \partial_t\zeta \Phi\circ u\dd x\dd t
    &=\int\limits_0^T\int\limits_{\mathbb{R}^d} \partial_t\zeta \int\limits_{K}^{\infty}\chi_{\{u>s\}}\Phi'(s)\dd s\dd x\dd t
    +\int\limits_0^T\int\limits_{\mathbb{R}^d} \partial_t\zeta \Phi(K)\dd x\dd t\\
    &=\int\limits_{K}^{\infty}\Phi'(s)\int\limits_0^T\int\limits_{\Omega_s(t)} \partial_t\zeta\dd x\dd t \dd s+0.
\end{align*}
Now, since $\Phi$ was arbitrary under the above assumptions, one can use $\Phi(s)=\int\limits_{-\infty}^s\varphi(s')\dd s'$ for $\varphi\in C_b^\infty(\mathbb{R}; \mathbb{R}_{> 0})$ a bump-function. Then, by the fundamental lemma of the calculus of variations, one has for almost every level set value $s$ (for $s<u_{\text{min}}$ both sides equal $0$) and for all test functions $\zeta\in C_c^1(\mathbb{R}^d\times (0,T))$
\begin{align*}
    \int\limits_0^T \int\limits_{\Omega_s(t)}\partial_t\zeta\dd x\dd t=-\int\limits_0^T\int\limits_{\Sigma_s(t)}\zeta V\dd\mathcal{H}^{d-1}(x)\dd t-\int\limits_{\Omega_s(0)} \zeta(\cdot,0)\dd x.
\end{align*}

\emph{Argument for \eqref{lvlMC}}: Fix again $\Phi\in C^\infty(\mathbb{R})$ with $\Phi'> 0$ and well-prepared initial data $g$. By Lemma~\ref{Relabeling}, $\Phi\circ g$ is also well-prepared.

Then, by \eqref{1distMC} using $\nu_\Phi=-\frac{\nabla (\Phi\circ u)}{|\nabla (\Phi\circ u)|}=\nu$ in $\{|\nabla u|\neq 0\}$ and $|\nabla (\Phi\circ u)|=\Phi'\circ u |\nabla u|$, we obtain:
\begin{align*}
\int\limits_0^T\int\limits_{\mathbb{R}^d}&\Phi'\circ u\left(\nabla \cdot \xi-\nu\cdot \nabla\xi\nu\right)|\nabla u|\dd x\dd t&=-\int\limits_0^T\int\limits_{\mathbb{R}^d}\Phi'\circ u \, \xi\cdot\nu V|\nabla u|\dd x\dd t
\end{align*}
and by the coarea-formula, this can be written as
\begin{align*}
    \int\limits_\mathbb{R} \Phi'(s)&\int\limits_0^T\int\limits_{\Sigma_s(t)}(\nabla \cdot \xi-\nu\cdot \nabla\xi\nu)\dd \mathcal{H}^{d-1}(x)\dd t\dd s\\
    &=-\int\limits_\mathbb{R}\Phi'(s)\int\limits_0^T\int\limits_{\Sigma_s(t)} \xi\cdot\nu V\dd \mathcal{H}^{d-1}(x)\dd t\dd s.
\end{align*}
As before, for almost every $s$ this yields the desired equality
$$\int\limits_0^T\int\limits_{\Sigma_s(t)} (\nabla \cdot \xi-\nu\cdot \nabla\xi\nu)\dd \mathcal{H}^{d-1}(x)\dd t=-\int\limits_0^T\int\limits_{\Sigma_s(t)} \xi\cdot\nu V\dd \mathcal{H}^{d-1}(x)\dd t,$$
which concludes the proof.
\end{proof}

Similarly, one can prove the level set version of Theorem~\ref{DeGiorgiViscosity}. This version will become the sharp energy dissipation for the BV solution.

\begin{thm}
Fix two time instances $0\leq t_1<t_2\leq T$. Then, for almost every level set value $s\in\mathbb{R}$, the sharp energy dissipation holds:
\begin{align}
    \mathcal{H}^{d-1}(\Sigma_s(t_2))+\int\limits_{t_1}^{t_2}\int\limits_{\Sigma_s(t)}V^2\dd \mathcal{H}^{d-1}(x)\dd t\leq \mathcal{H}^{d-1}(\Sigma_s(t_1)).
    \label{lvlphiDeriv}
\end{align}
\label{lvlphiDerivthm}
\end{thm}

\begin{proof}
This proof is similar to the previous theorem. We use the coarea-formula in Equation~\eqref{phiDeriv} to show that for almost every level set the difference of the sides of the inequality has a value of less or equal to zero:

By the relabeling property for $\Phi\in C^\infty(\mathbb{R})$ with $\Phi'\geq 0$ applied to Equation~\eqref{phiDeriv} one gets:
\begin{align*}
    0&\geq \left.\int\limits_{\mathbb{R}^d} \Phi'\circ u|\nabla u|\dd x\right|_{t_2}\left.+\int\limits_{t_1}^{t_2}\int\limits_{\mathbb{R}^d}\Phi'\circ uV^2|\nabla u|\dd x\dd t-\int\limits_{\mathbb{R}^d}\Phi'\circ u|\nabla u|\dd x\right|_{t_1}\\
    &=\int\limits_\mathbb{R}\Phi'(s)\left(\mathcal{H}^{d-1}(\Sigma_s(t_2))+\int\limits_{t_1}^{t_2}\int\limits_{\Sigma_s} V^2\dd\mathcal{H}^{d-1}(x)\dd t-\mathcal{H}^{d-1}(\Sigma_s(t_1))\right)\dd s.
\end{align*}
Since $\Phi'\geq 0$ and by the fundamental lemma of the calculus of variations, this yields the desired inequality.
\end{proof}

To prove now that almost every level set is a distributional solution, some well-posedness conditions have to be satisfied. The following lemma shows that our sets are sets of finite perimeter and that $u$ is differentiable a.e.\ on these, which is stated in the following lemma that can be found in \cite{ESIV} as Lemma 6.1.

\begin{lemma}
For every time $0\leq t\leq T$ and almost every level set value $s\in\mathbb{R}$ the viscosity solution $u$ is differentiable $\mathcal{H}^{d-1}$-almost everywhere on $\Sigma_s$ with non-vanishing gradient $|\nabla u|\neq 0$.\label{aeDiff}
\end{lemma}
\begin{proof}
For fixed $0\leq t\leq T$, the function $u$ is Lipschitz and thus by the coarea-formula for $\zeta:\mathbb{R}^d\to \mathbb{R}$ integrable one has
\begin{align*}
    \int\limits_{\mathbb{R}^d}\zeta |\nabla u|\dd x&=\int\limits_{\mathbb{R}}\int\limits_{\Sigma_s} \zeta\dd\mathcal{H}^{d-1}\dd s.
\end{align*}
For $\zeta=\chi_{A}, A:=\{u \text{ not diff. or } \nabla u=0\}$ it yields $\mathcal{H}^{d-1}(A\cap\Sigma_s)=0$.
\end{proof}

The last ingredient which is needed to prove Theorem~\ref{Thm:MainThm} is a way to connect the null sets that come from Theorem~\ref{lvlphiDerivthm} and Lemma~\ref{aeDiff}. This is necessary to achieve the optimal energy dissipation in \eqref{def:BVEnergyDissipation} for almost any time $T'$ and is done using the following lemma.

\begin{lemma}[$L^1$-continuity in time of BV solutions]
    Let $(\Omega(t))_{t\in [0,T)}$ be a distributional solution to mean curvature flow according to Definition~\ref{Def:BVsol}. Then, $\Omega_t$ is continuous in time in the $L^1$-norm for almost every time, i.e., for a.e.\ $t\in[0,T)$ there exists a sequence $t_n\to t$ s.t.\ $|\Omega_{t_n}\Delta\Omega_t|\to 0$.
    \label{Lemma:L1Cont}
\end{lemma}

\begin{proof}
    By approximation with $C^1_c(\mathbb{R}^d\times[0,T))$-functions, we can plug in the following $C_c(\mathbb{R}^d\times[0,T))$ function since $\Omega(t)$ is uniformly bounded:
    $$\zeta(t,x)\coloneqq \zeta_{t_0,n}(t)\eta(x), \zeta_{t_0,n}(t)\coloneqq\begin{cases}
        1 &\text{for }0\leq t\leq t_0,\\
        1-n(t-t_0) &\text{for }t_0\leq t\leq t_0+\frac{1}{n},\\
        0 &\text{else}.
    \end{cases}$$
    and $\eta\in C_c^\infty(\mathbb{R}), 0\leq \eta\leq 1$. Plugging in this test function in \eqref{Eq:LSdistV}, we achieve:
    \begin{align*}
        \fint\limits_{t_0}^{t_0+\frac{1}{n}}\int\limits_{\Omega(t)}\eta(x)\dd x\dd t&=\int\limits_{\mathbb{R}^d\times(0,T)}\zeta_{t_0,n}(t)\eta(x)V(x,t)\dd|\mu_\Omega(t)|\dd t+\int\limits_{\Omega(0)}\zeta(x,0)\dd x.
    \end{align*}
    By the Lebesgue point theorem, the left-hand side converges for almost every time $t_0$ as $n\to \infty$. Thus, by subtracting this equation with two such times $t_1$ and $t_0$, $t_1\geq t_0$, we get:
    \begin{align*}
        \Bigg|\int\limits_{\Omega(t_0)}\eta(x)\dd x&-\int\limits_{\Omega(t_1)}\eta(x)\dd x\Bigg|\\
        &\leq \limsup\limits_{n\to \infty}\left|\int\limits_{t_0}^{t_1+\frac{1}{n}}(\zeta_{t_0,n}-\zeta_{t_1,n})\int\limits_{\mathbb{R}^d}\eta(x)V(x,t)\dd|\mu_\Omega(t)|\dd t\right|\\
        &\leq \sqrt{t_1-t_0}\left(\int\limits_{\mathbb{R}^d\times (0,T)}V^2\dd |\mu_\Omega(t)|\dd t\right)^\frac{1}{2}.
    \end{align*}
    By \eqref{Eq:LSV}, the right-hand side is bounded. Hence, we can choose a sequence s.t.\ $t_1\to t_0$. This implies the continuity when one considers the supremum over all $\eta$:
    $$|\Omega(t_0)\Delta \Omega(t_1)|\to 0$$
    when $t_1\to t_0$ for almost every $t_0\in [0,T)$.
\end{proof}

\begin{proof}[Proof of Theorem~\ref{Thm:MainThm}]

Let $u\in C(\mathbb{R}^d\times (0,T))$ be a viscosity solution to the level set equation~\eqref{Eq:lvleq} of mean curvature flow with well-prepared initial data $g\in C^3(\mathbb{R}^d)$. 
Then, by Inequality~\eqref{lvlphiDeriv} in Theorem~\ref{lvlphiDerivthm} the $L^2$-bound \eqref{Eq:LSV} of $V$ is satisfied (when plugging in $t_1=0$ and $t_2=T$). This is due to the fact that for almost all level sets the reduced boundary of the super level sets is the level set $\partial^*\Omega_s(t)=\Sigma_s(t)$.

Moreover, the distributional characterizations~\eqref{Eq:LSdistV} and~\eqref{LSdistMC} of the normal velocity $V$ and the mean curvature $H$ were proven in Theorem~\ref{lvlVariants} in Equations~\eqref{lvlV} and \eqref{lvlMC}, respectively.

So far, everything is true up to a null sets of level set values as we only used the previous theorems for a finite amount of time values. Next, we want to show the sharp energy dissipation \eqref{def:BVEnergyDissipation}. For this, we can plug in $t_1=0, t_2=T'_n$ into Inequality~\eqref{lvlphiDeriv} in Theorem~\ref{lvlphiDerivthm} for a dense countable amount of final times $T'_n$.
Combining this with Lemma~\ref{Lemma:L1Cont} and the lower semi continuity of the perimeter, we can infer \eqref{def:BVEnergyDissipation} by considering a sequence $T'_n\searrow T'$:
\begin{align*}
P(\Omega(T'))&+\int\limits_0^{T'}\int\limits_{\mathbb{R}^d}V^2\dd|\mu_\Omega(t)|\dd t\\
&\leq \liminf\limits_{n\to\infty}P(\Omega(T'_n))+\int\limits_0^{T'}\int\limits_{\mathbb{R}^d}V^2\dd|\mu_\Omega(t)|\dd t\leq P(\Omega(0)).
\end{align*} 
Moreover, we can infer \eqref{Eq:LSP} due to the monotonicity in a similar way. For this, we again take the dense countable set of times in which $\Omega(t)$ is $L^1$-continuous according to Lemma~\ref{Lemma:L1Cont} and $P(\Omega(t))=\mathcal{H}^{d-1}(\Sigma(t))$. By the lower semi continuity of the perimeter, we can use Inequality~\eqref{lvlphiDeriv} to conclude the uniform perimeter bound \eqref{Eq:LSP}.
This proves the claim.
\end{proof}

\section{Extensions and open problems}
\subsection{Extension to homogeneous Neumann boundary conditions}
\label{homogNeumann}

In this paper, we worked in the setting of $\mathbb{R}^d$ with initial data which are constant outside a ball. One could also consider the case of homogeneous Neumann data, i.e., the surface $\Sigma$ satisfies $V=-H$ and is orthogonal to a given fixed domain boundary $\Sigma\perp \partial D$.

Similarly to before, one can define level set solutions as well as viscosity solutions and an analogous vanishing viscosity approximation, cf.~\cite{SatoruNeumann}. 
The level set equation becomes
\begin{align}
    \begin{cases}
        \partial_t u=\Delta u-\frac{\nabla u}{|\nabla u|}\cdot \nabla^2 u\frac{\nabla u}{|\nabla u|}   &\text{in }D\times (0,T),\\
        \nabla u\cdot \nu_{\partial D}=0&\text{on }\partial D\times (0,T),\\
        u=g&\text{on }D\times \{t=0\}.
    \end{cases}
    \label{Eq:level setEquationNeumann}
\end{align}
Here, the initial data $g$ also has to satisfy the boundary condition $\nabla g\cdot\nu_{\partial D}=0$. Now, the vanishing viscosity approximation can be constructed in the same way satisfying the condition $\nabla u_\varepsilon\cdot \nu_{\partial D}=0$. 
The major difference in this context lies in the formulation of the distributional solution. 
To encode the boundary conditions, we use again integration by parts, see also~\cite{HenselLaux_Neumann}.
Indeed, the homogeneous Neumann condition $\nu_{\partial D} \cdot \nabla u =0$ can be encoded by asking that for any test vector field $\xi\in C_c^1(\overline{D}\times (0,T);\mathbb{R}^d)$ with $\xi\cdot\nu_{\partial D}=0$ it holds
\begin{align*}
    \int\limits_0^T\int\limits_{\mathbb{R}^d}\left(\nabla \cdot \xi-\nu_u\cdot \nabla\xi\nu_u\right)|\nabla u|\dd x\dd t=-\int\limits_0^T\int\limits_{\mathbb{R}^d} \xi\cdot\nu_u V|\nabla u|\dd x\dd t.
\end{align*}

\begin{thm}[Viscosity solutions are distributional solutions --- Neumann conditions]
    Let $D$ be a smooth domain in $\mathbb{R}^d$ and $u\in C(D\times [0,T))$ be a viscosity solution to the level set mean curvature flow with Neumann condition \eqref{Eq:level setEquationNeumann} starting with well-prepared initial data $g$. I.e., $g$ satisfies Definition~\ref{def:wellprepared} and $\nabla g\cdot \nu_{\partial D}=0$. 
    Then, almost every level set of $u$ is a BV solution to mean curvature flow in the sense of Definition~\ref{Def:BVsol} where the assumption on $\xi$ for Equation~\eqref{LSdistMC} is replaced by $\xi\in C_c^1(\overline{D}\times (0,T);\mathbb{R}^d)$ with $\xi\cdot\nu_{\partial D}=0$.
\end{thm}

The weak formulation follows from the following calculation (in which we suppress the $t$ dependence) for the vanishing viscosity approximation
\begin{align*}
    -\int\limits_{D} \xi\cdot\nu_\varepsilon \partial_tu_\varepsilon\dd x&=-\int\limits_{D} \xi\cdot\nabla u_\varepsilon\nabla\cdot \nu_\varepsilon\dd x\\
    &=\int\limits_{D} \nu_\varepsilon\cdot \nabla\xi \nabla u_\varepsilon \dd x+\int\limits_{D} \xi \cdot \nabla^2 u_\varepsilon \nu_\varepsilon \dd x\\
    &\quad-\int\limits_{\partial D}\xi\cdot\nabla u_\varepsilon \nu_\varepsilon \cdot \nu_{\partial D}\dd \mathcal{H}^{d-1}(x)\\
    &=\int\limits_{D} \nu_\varepsilon\cdot \nabla\xi \nabla u_\varepsilon \dd x
    -\int\limits_{D} \xi \cdot \nabla\sqrt{|\nabla u_\varepsilon|^2+\varepsilon^2} \dd x\\
    &=\int\limits_{D} \nu_\varepsilon\cdot \nabla\xi \nabla u_\varepsilon \dd x+\int\limits_{D} \nabla\cdot\xi  \sqrt{|\nabla u_\varepsilon|^2+\varepsilon^2} \dd x\\
    &\quad-\int\limits_{\partial D}\xi \cdot \nu_{\partial D}\sqrt{|\nabla u_\varepsilon|^2+\varepsilon^2}\dd \mathcal{H}^{d-1}(x)\\
    &=\int\limits_{D}\left(\nabla \cdot \xi-\nu_\varepsilon\cdot \nabla\xi\nu_\varepsilon\right)\sqrt{|\nabla u_\varepsilon|^2+\varepsilon^2}\dd x\\
    &\quad-\int\limits_{\partial D}\xi\cdot\nu_{\partial D}\sqrt{|\nabla u_\varepsilon|^2+\varepsilon^2}\dd \mathcal{H}^{d-1}(x).
\end{align*}
For test vector fields $\xi$ with $\xi\cdot \nu_{\partial D}=0$, the additional boundary term vanishes exactly. Then, one can pass to the limit $\varepsilon\searrow 0$.
With these changes, we can modify the proof to show Theorem~\ref{Thm:MainThm} for homogeneous Neumann conditions. 
Due to~\cite[Lemma 2.3]{SatoruNeumann}, the existence of well-prepared initial data to a $0$-level set which meets $\partial D$ orthogonally is ensured.
Moreover, in this paper, the integrability of the approximate mean curvature as well as the convergences and the $L^2$-integrability with coarea-factor is shown, cf.~\cite[Theorem 3.2, Lemma 3.3 and Lemma 3.4]{SatoruNeumann}. Theorems~\ref{distViscosity} and \ref{DeGiorgiViscosity} follow with the same reasoning with the sole difference that the convergences now use that we are on the bounded domain $D$.

\subsection{Further directions and open questions}

We showed that almost every level set of the viscosity solution is a BV solution. 
It remains an interesting open question whether the inverse assertion holds.
Let $g$ be well-prepared initial data and $u\in C(\mathbb{R}^d\times (0,T))$ such that every level set of $u$ is a distributional solution to mean curvature flow. Is $u$ then a viscosity solution to mean curvature flow? By the weak-strong uniqueness result in~\cite{FHLS}, this is true as long as a classical solution exists. However, after the onset of singularities this is still an open question.

A further interesting generalization would be to incorporate non-ho\-mo\-ge\-neous Neumann condition, i.e., $\frac{\nabla u}{|\nabla u|}\cdot \nu_D=\cos \alpha$. 
Geometrically speaking, this fixes the angle between the free interface and the boundary $\partial D$. 
The angle condition is the result of a boundary energy, which enforces the angle condition via Young's law. 
For similar results regarding the phase-field approximation, we refer to~\cite{HenselLaux_Neumann}.

It seems feasible to generalize our present proof to the anisotropic case, just as Tonegawa~{\cite{TonegawaAnisotropic}} generalized the proof of Evans and Spruck~\cite{ESIV} to the anisotropic setting.

\section{Appendix}

The following theorem is taken from~\cite[Part 2]{MaggiSofP} and gives a characterization of the distributional derivative of characteristic functions.
\begin{thm}[De Giorgi's structure theorem]
\label{DeGiorgiStructure}
Let $A$ be a set of finite perimeter. Then the distributional derivative of its characteristic function can be expressed using the generalized outer unit normal vector and the reduced boundary via $\nabla\chi_A=\nu^A\mathcal{H}^{d-1}\lfloor_{\partial^*A}$.
\end{thm}

The coarea-formula can also be found in~\cite[Chapter 13]{MaggiSofP}.
\begin{lemma}[Coarea-formula]
\label{coareaformula}
Let $u:\mathbb{R}^d\to \mathbb{R}$ be Lipschitz, $\Omega\subseteq\mathbb{R}^d$ open and $f\in L^1(\mathbb{R}^d)$ integrable. Then
\begin{align*}
\int\limits_\Omega |\nabla u(x)|f(x)\dd x&=\int\limits_\mathbb{R}\int\limits_{\{x\in\Omega:u(x)=s\}}f(x)\dd\mathcal{H}^{d-1}(x)\dd s.
\end{align*}
\end{lemma}

\section*{Statements and Declarations}

The present paper is an extension of the second author's master's thesis at the  University of Bonn.
This project has received funding from the Deutsche Forschungsgemeinschaft (DFG, German Research Foundation) under Germany's Excellence Strategy --- EXC-2047/1 --- 390685813.

\frenchspacing
\bibliographystyle{abbrv}
\bibliography{Referenzen.bib} 
\end{document}